\documentclass[11pt,amssymb]{amsart}
\usepackage{amssymb}
\usepackage[mathscr]{eucal}

\newcommand{\Z}{{\mathbb Z}}
\newcommand{\Gu}{\underline{G}}

\newcommand{\Q}{{\mathbb Q}}

\newcommand{\R}{{\mathbb R}}

\newcommand{\Ga}{\mathrm{Gal}}
\newtheorem{thm}{Theorem}[section]
\newtheorem{lemma}[thm]{Lemma}
\newtheorem{prop}[thm]{Proposition}

\newcommand{\no}{\vert\vert}

\begin{document}

\title[Strong Approximation]{On  strong approximation for algebraic groups}

\author[Andrei Rapinchuk]{Andrei S. Rapinchuk}

\address{Department of Mathematics, University of Virginia,
Charlottesville, VA 22904-4137, USA}

\email{asr3x@virginia.edu}

\maketitle

\section{Introduction}\label{S:I}

The goal of this article, which is an expanded version of the talk
given at the workshop, is to provide a survey of known results
related to strong approximation in algebraic groups. We will focus
primarily on two aspects: the classical form of strong
approximation, which is really strong approximation for
$S$-arithmetic groups (\S \ref{S:AG}), and its more modern version
for arbitrary Zariski-dense subgroups (\S \ref{S:ZD}). Along the way
we will also mention results dealing with strong approximation in
arbitrary varieties and particularly homogeneous spaces (which are
probably not so well known to the general audience as some other
results in the article) and  some applications. The reader will find
more applications of  strong approximation for Zariski-dense
subgroups in other articles in this volume.

\vskip2mm

\noindent {\bf 1. Strong approximation and congruences.} The most
elementary way to start thinking about strong approximation is in
terms of lifting solutions of integer polynomial equations
$\mathrm{mod}\: m$ for all $m \geqslant 1$, to integer solutions.
So, suppose we have a family of polynomials
$$
f_{\alpha}(x_1, \ldots , x_d) \in \Z[x_1, \ldots , x_d], \ \ \
\alpha \in I,
$$
and we let $X \subset \mathbb{A}_{\Z}^d$ denote the closed affine
subscheme defined by these polynomials. Thus, for any $\Z$-algebra
$R$, the scheme $X$ has the following set of $R$-points:
$$
X(R) = \{ (a_1, \ldots , a_d) \in R^d \ \vert \ f_{\alpha}(a_1,
\ldots , a_d) = 0 \ \ \text{for all} \ \ \alpha \in I \}.
$$
Then for any integer $m \geqslant 1$, we have a natural reduction
modulo $m$ map
$$
\rho_m \colon X(\Z) \to X(\Z/m\Z),
$$
and the question is whether these maps are {\it surjective} for all
$m$. (Of course, this question is meaningful only if we assume that
$X(\Z/m\Z) \neq \emptyset$ for all $m$.) Observe that for $m \mid
n$, there is a canonical homomorphism $\Z/n\Z \to \Z/m\Z$, hence a
natural map $\pi^n_m \colon X(\Z/n\Z) \to X(\Z/m\Z)$. Clearly, $\{
\: X(\Z/m\Z) \: , \: \pi^n_m \: \}$ is an inverse system, so we can
assemble all the $X(\Z/m\Z)$'s together by taking the inverse limit:
$$
\lim_{\longleftarrow} X(\Z/m\Z) = X(\hat{\Z}),
$$
where $\displaystyle \hat{\Z} = \lim_{\longleftarrow} \Z/m\Z$.
Recall that the Chinese Remainder Theorem furnishes an isomorphism
$\hat{\Z} \simeq \prod_p \Z_p$, where $\Z_p$ is the ring of $p$-adic
integers and the product is taken over all primes, which allows us
to identify $X(\hat{\Z})$ with $\prod_p X(\Z_p)$.

Just as above, for any integer $m \geqslant 1$, there is a natural
map
$$
\hat{\rho} \colon X(\hat{\Z}) \longrightarrow X(\hat{\Z}/m\hat{\Z})
= X(\Z/m\Z).
$$
The pre-images of points under the $\hat{\rho}_m$'s form a basis of
a natural topology on $X(\hat{\Z})$, which coincides with either of
the following topologies:

\vskip1mm

$\bullet$ \parbox[t]{11.5cm}{the topology of the inverse limit on
$\displaystyle \lim_{\longleftarrow} X(\Z/m\Z)$, cf. \cite[Ch.I,
5.3]{KNV};}

\vskip1mm

$\bullet$ \parbox[t]{11.5cm}{the topology induced by the embedding
$X(\hat{\Z}) \hookrightarrow \hat{\Z}^d$, where $\hat{\Z}$ is
endowed with the inverse limit topology on $\displaystyle
\lim_{\longleftarrow} \Z/m\Z$;}

\vskip1mm

$\bullet$ \parbox[t]{11.5cm}{the direct product topology on $\prod_p
X(\Z_p)$, where $X(\Z_p)$ receives its topology via the embedding
$X(\Z_p) \hookrightarrow \Z_p^d$, and $\Z_p$ is endowed with the
natural $p$-adic topology.}

\vskip2mm

\noindent The following immediately follows from the above
discussion.
\begin{lemma}\label{L:1}
The following conditions are equivalent:

\vskip1mm

{\rm (1)} $\rho_m \colon X(\Z) \to X(\Z/m\Z)$ is surjective for all
integers $m \geqslant 1$;

\vskip1mm

{\rm (2)} \parbox[t]{11cm}{the natural embedding $\iota \colon X(\Z)
\hookrightarrow X(\hat{\Z})$ has a dense image in the above
topology.}
\end{lemma}

In this situation, we say $X$ has {\it strong approximation} if it
satisfies the equivalent conditions of Lemma \ref{L:1} (of course,
this is only a first approximation to the precise definition(s) of
strong approximation that will be given later, cf. \S~2.1).
Intuitively, strong approximation should not be very common as there
are plentiful examples where $X(\hat{\Z}) \neq \emptyset$ but $X(\Z)
= \emptyset$ (i.e., the Hasse principle fails - note that here we
actually omit the archimedean place of $\Q$), and also examples
where $X(\Z)$ is nonempty but so ``small" that it cannot possibly be
dense in $X(\hat{\Z})$. A classical example of the second situation
is a cubic hypersurface $X \subset \mathbb{A}^3$ given by the
equation
$$
3x^3 + 4y^3 + 5z^3 = 0;
$$
it is known that $X(\Z) = \{(0, 0, 0)\}$ but $X(\Z_p) \neq \{ (0, 0,
0) \}$ (hence infinite as any point on $X$ other than the origin is
smooth) for all prime $p$. In fact, very little appears to be known
about strong approximation for schemes (varieties) lying outside
some special classes such as homogeneous spaces - one can only give
some {\it necessary conditions} (cf. Proposition \ref{P:Minchev} and
subsequent remarks). So, in this article we will deal almost
exclusively with algebraic groups.

\vskip2mm

\noindent {\bf 2. $\mathbf{\mathrm{SL}_2}$ vs.
$\mathbf{\mathrm{GL}_2}$.} Let us start with two elementary
examples: $G_1 = \mathrm{SL}_2$ and $G_2 = \mathrm{GL}_2$. One
doesn't see much of a difference between these examples just by
looking at the defining equations. Indeed, with the obvious labeling
of coordinates, we see that

\vskip2mm

\noindent $\bullet$ \parbox[t]{12cm}{$G_1$ can be realized as a
hypersurface in $\mathbb{A}^4$ given by $x_{11}x_{22} - x_{12}x_{21}
= 1$; \vskip1mm  and}

\vskip2mm

\noindent $\bullet$ \parbox[t]{12.3cm}{$G_2$ can be realized as a
hypersurface in $\mathbb{A}^5$ given by $y(x_{11}x_{22} -
x_{12}x_{21}) = 1$.}

\vskip2mm

\noindent However, $G_1$ has strong approximation, and $G_2$ does
not.

\vskip1mm

\begin{lemma}\label{L:2}
For any $m > 1$, the reduction modulo $m$ map
$$
\rho_m \colon SL_2(\Z) \longrightarrow SL_2(\Z/m\Z)
$$
is surjective.
\end{lemma}
\begin{proof}
The argument does not use equations (in fact, it is not a completely
trivial task to prove strong approximation using equations in this
case - see the discussion after Proposition \ref{P:HS1}). The
crucial observation is that any $\bar{g} \in SL_2(\Z/m\Z)$ can be
written as a product of elementary matrices:
\begin{equation}\label{E:1}
\bar{g} = \prod_k e_{i_k j_k}(\bar{a}_k) \ \ \text{with} \ \ (i_k \;
, \; j_k) \in \{(1 \; , \; 2) \: , \: (2 \; , \; 1)\} \ \ \text{and}
\ \ \bar{a}_k \in \Z/m\Z.
\end{equation}
(As usual, for $i \neq j$, we let $e_{ij}(a)$ denote the elementary
matrix having $a$ as its $ij$-entry.) For this, one needs to observe
that if $m = p_1^{\alpha_1} \cdots p_r^{\alpha_r}$ then it follows
from the Chinese Remainder Theorem that
$$
SL_2(\Z/m\Z) = SL_2(\Z/p_1^{\alpha_1}\Z) \times \cdots \times
SL_2(\Z/p_r^{\alpha_r}\Z),
$$
which reduces the problem to the case where $m = p^{\alpha}$. Now,
given
$$
\bar{g} = \left( \begin{array}{cc} x_{11} & x_{12} \\ x_{21} &
x_{22} \end{array} \right) \ \in \ SL_2(\Z/p^{\alpha}\Z),
$$
we see that either $x_{11}$ or $x_{12}$ is a unit $\mathrm{mod}\:
p^{\alpha}$, so using Gaussian elimination one can easily write
$\bar{g}$ as a product of elementaries over $\Z/p^{\alpha}\Z$.

Next, given an arbitrary  $\bar{g} \in SL_2(\Z/m\Z)$, pick a
factorization (\ref{E:1}), and furthermore, for each $k$ pick an
integer $a_k$ in the class $\bar{a}_k$ modulo $m$. Set
$$
g = \prod_k e_{i_k j_k}(a_k) \in SL_2(\Z).
$$
Then $\rho_m(g) = \bar{g}$, proving the surjectivity of $\rho_m$.
\end{proof}

Note that the proof of Lemma \ref{L:2} relies on the consideration
of unipotent elements, so it is worth pointing out  that, as we will
see in the course of this article, unipotent elements are involved
in one way or another in most known results on strong approximation
(even when the group at hand does not contain any nontrivial
rational unipotent elements, i.e. is anisotropic).

\vskip2mm

The fact that $G_2 = \mathrm{GL}_2$ does not have strong
approximation is much easier: in fact, already the map
$$
\rho_5 \colon GL_2(\Z) \longrightarrow GL_2(\Z/5\Z)
$$
fails to be surjective. (Indeed, since all matrices in $GL_2(\Z)$
have determinant $\pm 1$, the matrices in $\rho_5(GL_2(\Z))$ have
determinant $\pm 1(\mathrm{mod}\: 5)$, and therefore, for example,
$\left( \begin{array}{cc} \bar{1} & \bar{0} \\ \bar{0} & \bar{2}
\end{array} \right) \in GL_2(\Z/5\Z)$ does not lie in the image of
$\rho_5$.) One can conceptually articulate the obstruction that
prevents $\mathrm{GL}_2$ from having strong approximation in this
case by saying that in order for an affine $\Q$-variety $X$ to have
strong approximation,

\vskip2mm

\centerline{$X(\Z)$ \ must \ be \ Zariski-dense \ in \ $X$.}

\vskip2mm

\noindent Indeed, let $Y = \overline{X(\Z)}$ be the Zariski-closure
of $X(\Z)$ in $X$, and assume that $Y \neq X$. Pick a point $a \in
X(\bar{\Q}) \setminus Y(\bar{\Q})$, where $\bar{\Q}$ is an algebraic
closure of $\Q$. Then one can find a polynomial $f \in \Z[x_1,
\ldots , x_d]$ that vanishes on $Y$ and such that $f(a) \neq 0$. It
follows from Tchebotarev's Density Theorem that for infinitely many
primes $p$, we have $a \in X(\Z_p)$ and $f(a) \not\equiv
0(\mathrm{mod}\: p)$. Let $\bar{a} \in X(\mathbb{F}_p)$ be the
reduction of $a$ modulo $p$, where $\mathbb{F}_p = \Z/p\Z =
\Z_p/p\Z_p$. (Note that it would be more appropriate to write
$\underline{X}^{(p)}(\mathbb{F}_p)$ instead of $X(\mathbb{F}_p)$,
where $\underline{X}^{(p)}$ denotes the reduction of $X$ modulo $p$,
but we will slightly abuse the notations in this introductory
section in order to keep them simple.) Then clearly
$$
\bar{a} \in X(\mathbb{F}_p) \setminus Y(\mathbb{F}_p),
$$
and in particular, $X(\mathbb{F}_p) \neq Y(\mathbb{F}_p)$. On the
other hand, the image of the reduction map $\rho_p \colon X(\Z) \to
X(\mathbb{F}_p)$ is obviously contained in $Y(\mathbb{F}_p)$. Thus,
if $X(\Z)$ is not Zariski-dense  in $X$ then $\rho_p$ fails to be
surjective for infinitely many $p$, which certainly prevents $X$
from having strong approximation. (Incidentally, this observation
implies that if $G$ is an algebraic $\Q$-group and $G(\Z)$ is not
Zariski-dense in $G$ then the closure of $G(\Z)$ in $G(\hat{\Z})$ is
of infinite index.)

In fact, the conclusion about the absence of strong approximation in
$X$ as above can be made even sharper. First, it is easy to show
that $X$ cannot possibly have strong approximation unless it is {\it
absolutely irreducible} (cf. the remark after Proposition
\ref{P:Minchev}). So, assume that $X$ is such. Then by the Lang-Weil
estimates (cf. \cite{LaWe}) we have
$$
\vert X(\mathbb{F}_p) \vert \approx p^{\dim X}
$$
for $p$ sufficiently large. Similarly, for any proper
$\Q$-subvariety $Y \subset X$,  the cardinality $\vert
Y(\mathbb{F}_p) \vert$ is bounded above by an expression of the form
$C \cdot p^{\dim Y}$ where $C$ is a constant independent of $p$. It
follows that $Y(\mathbb{F}_p) \neq X(\mathbb{F}_p)$ for almost all
$p$, and therefore unless $X(\Z)$ is Zariski-dense in $X$, the
reduction map $\rho_p \colon X(\Z) \to X(\mathbb{F}_p)$ is not
surjective for {\it almost all} $p$.

\vskip2mm

So, the fact that $GL_2(\Z)$ is not Zariski-dense in $\mathrm{GL}_2$
(its Zariski-closure is precisely the subgroup consisting of $g \in
\mathrm{GL}_2$ that satisfy $(\det g)^2 - 1 = 0$), is definitely one
of the factors that prevent $\mathrm{GL}_2$ from having strong
approximation; in fact, the reduction maps $\rho_p$ are
nonsurjective for all $p \geqslant 5$. Now, let us slightly change
the set-up by replacing the ring of integers $\Z$ with some
localization, e.g. $\Z\left[\frac{1}{2}\right]$. Then $
GL_2\left(\Z\left[\frac{1}{2}\right]\right)$ is already
Zariski-dense in $GL_2$, and in fact the map
$$
\rho_5 \colon GL_2\left(\Z\left[\frac{1}{2}\right] \right)
\longrightarrow GL_2(\Z/5\Z)
$$
is  surjective, however the map
$$
\rho_{17} \colon GL_2\left(\Z\left[\frac{1}{2}\right] \right)
\longrightarrow GL_2(\Z/17\Z)
$$
is not. The reason is that the possible determinants of matrices in
$GL_2\left(\Z\left[\frac{1}{2}\right]\right)$ are of the form $\pm
2^{\ell}$ with $\ell \in \Z$, hence squares modulo $p = 17$ (in
fact, this property will hold for any prime of the form $8k + 1$,
and by Dirichlet's Theorem there are infinitely many such primes,
cf. \S~2.2).

We see that Zariski-density is definitely not sufficient for strong
approximation in the general case. At the same time, let us consider
the following example involving various subgroups of the group
$SL_2(\Z)$. We have
$$
\Gamma_0 := SL_2(\Z) = \left\langle \ \left(\begin{array}{cc} 1 & 1 \\
0 & 1  \end{array} \right)  \ , \ \left(\begin{array}{cc} 1 & 0 \\
1 & 1 \end{array} \right) \ \right\rangle.
$$
For $\ell \geqslant 1$, we define
$$
\Gamma_{\ell} = \left\langle \ \left(\begin{array}{cc} 1 & 2^{\ell} \\
0 & 1  \end{array} \right)  \ , \ \left(\begin{array}{cc} 1 & 0 \\
2^{\ell} & 1 \end{array} \right) \ \right\rangle.
$$
Then we have the following inclusions
$$
\Gamma_0 \supset \Gamma_1 \supset \Gamma_2 \supset \cdots \supset
\Gamma_{\ell} \supset \Gamma_{\ell + 1} \supset \cdots  ,
$$
with
$$
[\Gamma_0 : \Gamma_1] = 12 \ \ \ \text{and} \ \ \ [\Gamma_{\ell} :
\Gamma_{\ell + 1}] = \infty \ \  \text{for}  \ \ \ell \geqslant 1.
$$
(We note that the fastest way to verify both of these claims is to
use the {\it virtual} Euler-Poincar\'e characteristic (cf.
\cite{Serre-cohom}). It is known that the Euler-Poincar\'e
characteristic $\chi(\Gamma_0) = - \frac{1}{12}$. On the other hand,
for any $m \geqslant 2$ the matrices $\left(
\begin{array}{cc} 1 & m \\ 0 & 1 \end{array}\right)$ and $\left(
\begin{array}{cc} 1 & 0 \\ m & 1 \end{array}\right)$ generate a
rank 2 free subgroup $\Delta_m \subset \Gamma_0$ (cf. \cite[p.
26]{delaHarpe}), so $\chi(\Delta_m) = -1$. It is an elementary
exercise to show that $\Gamma_1 = \Delta_2$ contains the congruence
subgroup $SL_2(\Z , 4)$ modulo 4, so the index $d = [\Gamma_0 :
\Gamma_1]$ is finite. So we have
$$
\chi(\Gamma_1) = d \cdot \chi(\Gamma_0),
$$
whence $d = 12$, as claimed. On the other hand, the assumption that
$[\Gamma_{\ell} : \Gamma_{\ell+1}] =: d < \infty$ would imply that
$$
-1 = \chi(\Gamma_{\ell+1}) = d \cdot \chi(\Gamma_{\ell}) = - d,
$$
i.e. $\Gamma_{\ell+1} = \Gamma_{\ell}$ which is clearly false
(consider the reduction modulo $2^{\ell + 1})$. Incidentally, the
same argument shows that $\Delta_m$ is of infinite index in
$\Gamma_0$ for {\it any} $m \geqslant 3$. Indeed, we  can now assume
that $m$ is not a power of $2$. If $[\Gamma_0 : \Delta_m] = d <
\infty$ then
$$
-1 = \chi(\Delta_m) = d \cdot \chi(\Gamma_0) = - \frac{d}{12},
$$
implying that $d = 12$. But $\Delta_m$ is contained in the
congruence subgroup $SL_2(\Z , m)$, so if $p$ is an odd prime
divisor of $m$ then
$$
[\Gamma_0 : \Delta_m] \geqslant [\Gamma_0 : SL_2(\Z , m)] \geqslant
\vert SL_2(\mathbb{F}_p) \vert = p(p^2 - 1) \geqslant 24,
$$
a contradiction. We note that the group $\Delta_3 = \left\langle
\left( \begin{array}{cc} 1 & 3 \\ 0 & 1 \end{array} \right) , \left(
\begin{array}{cc} 1 & 0 \\ 3 & 1 \end{array} \right) \right\rangle$
has received a lot of attention during the workshop.)

\vskip2mm

So, for large $\ell$, the subgroup $\Gamma_{\ell}$ is very ``thin"
in $\Gamma_0$, and essentially the only property it retains is
Zariski-density. Nevertheless, for all {\it odd} $m$ we still have
$$
\rho_m(\Gamma_{\ell}) = \rho_m(\Gamma_0) = SL_2(\Z/m\Z).
$$
So, if we ignore $p = 2$ (more precisely, the dyadic component
$\Z_2$ of $\hat{\Z}$), then we still have an analog of the property
of strong approximation for $\Gamma_{\ell}$, for {\it any} $\ell
\geqslant 1$. At the same time, the closure of $\Gamma_{\ell}$ in
$SL_2(\Z_2)$ is open (cf. Lemma  \ref{L:SA} for a more general
statement). Thus, we eventually obtain that the closure of
$\Gamma_{\ell}$ in $SL_2(\hat{\Z})$ is open -- one should think of
this property as being the next best thing to strong approximation.
Note that for a general $X$ as above, the openness of the closure of
$X(\Z)$ in $X(\hat{\Z})$ implies that the reduction maps $\rho_m
\colon X(\Z) \to X(\Z/m\Z)$ are surjective for all $m$ co-prime to
some fixed exceptional number $N_0 = N_0(X)$.

\vskip2mm

To summarize this discussion, we see that generally speaking the
idea that in certain situations Zariski-density should (or may)
imply some version of strong approximation, at least for subgroups,
appears to be sound, but in order to make it more precise, we need
to figure out what is wrong with $GL_2$ (compared to $SL_2$).

Before we do this, however, we would like to generalize our set-up
and also describe a somewhat different (although closely related)
approach to strong approximation. The issue is that typically an
algebraic group does not come with a fixed geometric (or linear)
realization $G \hookrightarrow \mathrm{GL}_n$, and different
realizations may result in different groups of integral points. So,
it makes sense to reformulate the property of strong approximation
in terms of the {\it group of rational points}.

\section{Strong approximation in algebraic groups and homogeneous
spaces}\label{S:AG}

\noindent {\bf 1. Adele groups and strong approximation.} Let $G$ be
an algebraic group defined over a global field $K$, and let $S$ be a
set of places of $K$. For now, we fix a matrix realization $G
\hookrightarrow \mathrm{GL}_n$, which enables us to define
unambiguously the groups
$$
G(\mathcal{O}_v) = G \: \bigcap \: GL_n(\mathcal{O}_v)
$$
for all nonarchimedean places $v$ of $K$, where $\mathcal{O}_v$ is
the valuation ring in the completion $K_v$. We let $\mathbb{A}_S$
denote the {\it ring of $S$-adeles} of $K$, and let
$$
G(\mathbb{A}_S) = \left\{ \: g = (g_v) \in \prod_{v \notin S} G(K_v)
\ \vert \ g_v \in G(\mathcal{O}_v) \ \ \text{for almost all} \ \ v
\notin S \right\}
$$
be the {\it group of $S$-adeles} of $G$. We refer the reader to
\cite[\S 5.1]{PlR} for a more detailed discussion of adeles, and in
particular for the definition of the space of $S$-adeles
$X(\mathbb{A}_S)$ for any affine algebraic $K$-variety $X$. Here we
recall only that $G(\mathbb{A}_S)$ is made into a locally compact
topological group by taking the open subgroups of $\prod_{v \notin
S} G(\mathcal{O}_v)$ for a fundamental system of neighborhoods of
the identity - thus, the $S$-adelic topology on $G(\mathbb{A}_S)$ is
the ``natural extension" of the product topology on $\prod_{v \notin
S} G(\mathcal{O}_v)$. (We note that in the case $K = \Q$, $S = \{
\infty \}$, the latter group coincides with $\prod_{p} G(\Z_p) =
G(\hat{\Z})$, so these adelic definitions are direct generalizations
of the notions we discussed in \S \ref{S:I}.) One proves (cf.
\cite[\S 5.1]{PlR}) that the topological group $G(\mathbb{A}_S)$ is
independent of the choice of a $K$-realization $G \hookrightarrow
\mathrm{GL}_n$. Furthermore, there is a canonical embedding $G(K)
\hookrightarrow G(\mathbb{A}_S)$, so we can give the following.

\vskip2mm

\noindent {\bf Definition.} An algebraic $K$-group $G$ has {\it
strong approximation} with respect to $S$ if $G(K)$ is dense in
$G(\mathbb{A}_S)$.

\vskip2mm

(Of course, one can give a similar definition for an arbitrary
affine $K$-variety $X$. We note that if $S = \emptyset$ then $X(K)$
is a closed discrete subspace of $X(\mathbb{A}_S)$, so in discussing
strong approximation one actually needs to assume from the outset
that $S$ is nonempty.)

\vskip2mm

Defined this way (in terms of rational points), the property of
strong approximation does not depend on the choice of a matrix
realization $G \hookrightarrow \mathrm{GL}_n$. On the other hand, in
the case where $S$ contains all nonarchimedean places, its validity
implies  that for {\it any} realization, the group
$G(\mathcal{O}(S))$ of points over the ring of $S$-integers
$\mathcal{O}(S)$, which can alternatively be described as
$$
G(\mathcal{O}(S)) = G(K) \: \bigcap \: \prod_{v \notin S}
G(\mathcal{O}_v),
$$
is dense in $\prod_{v \notin S} G(\mathcal{O}_v)$ (thus, we have
strong approximation in the sense discussed in \S \ref{S:I} for any
realization).

\vskip2mm

\noindent {\bf 2. Absence of strong approximation in algebraic
tori.} Our next goal is to explain why $\mathrm{GL}_2$ has no chance
to possess strong approximation. However, it is  easiest to pin down
the reason by working with the 1-dimensional $T = \mathbb{G}_m$: we
will now show that it does not have strong approximation with
respect to any finite set of places $S$, and will then demonstrate
how the same phenomenon manifests itself in the case of $GL_2$ and
other situations.

Let us start with the case $K = \Q$. If $S = \{ \infty \}$ then
$T(\Z) = \{ \pm 1 \}$ which is not even Zariski-dense. For $S = \{
\infty , 2 \}$, we have
$$
T\left(\Z\left[\frac{1}{2}\right]\right) = \pm \langle 2 \rangle,
$$
which is already Zariski-dense, but nevertheless $T$ still does not
have strong approximation. Indeed, pick any prime $p$ of the form
$8k + 1$. Then $-1$ and $2$ are squares modulo $p$, so the map
$$
\pm \langle 2 \rangle \to \left( \Z/p\Z \right)^{\times}
$$
is {\it not} surjective. What really happens here is that $T$
possesses a 2-sheeted cover
$$
\pi \colon T \to T, \ \ \ t \mapsto t^2,
$$
and for any prime $p \equiv 1(\mathrm{mod}\: 8)$ we have that
$$
T\left(\Z\left[ \frac{1}{2} \right]\right) \subset \pi(T(\Z_p))
\subsetneqq  T(\Z_p).
$$
Since $\pi(T(\Z_p)) \subset T(\Z_p)$ is a closed subgroup, we obtain
that $T\left(\Z\left[\frac{1}{2}\right]\right)$ is {\it not} dense
in $T(\Z_p)$ for any such $p$. Moreover, by Dirichlet's Prime Number
Theorem, for any $r \geqslant 1$ we can find $r$ distinct primes
$p_1, \ldots , p_r$ congruent to $1(\mathrm{mod}\: 8)$. Then the
image of the map
$$
\pm \langle 2 \rangle \to \left(\Z/p_1 \cdots p_r \Z
\right)^{\times}
$$
is contained in ${\left(\Z/p_1 \cdots p_r\Z\right)^{\times}}^2$,
which has index $2^r$ in $\left(\Z/p_1 \cdots
p_r\Z\right)^{\times}$. It follows that the closure of
$T\left(\Z\left[\frac{1}{2}\right]\right)$ in $T(\hat{\Z}) = \prod_p
T(\Z_p)$ is {\it of infinite index}.

This approach easily generalizes. First, let $T = \mathbb{G}_m$ over
an arbitrary number field $K$, and let $S$ be an arbitrary finite
set of places of $K$ containing all archimedean ones. Then by
Dirichlet's Unit Theorem (cf. \cite[p.~105]{Lang-NT}), the group
$T(\mathcal{O}(S))$ is generated by a finite collection of elements,
say $t_1, \ldots , t_r$. Set $L = K(\sqrt{t_1}, \ldots ,
\sqrt{t_r})$. Then by Tchebotarev's Density Theorem (cf.
\cite[p.~169]{Lang-NT}), there exist infinitely many places $v
\notin S$ that totally split in $L$ (i.e., $L \subset K_v$).
Considering again the covering $\pi \colon T \to T$, $\pi(t) = t^2$,
we see that that for any such nondyadic $v$ we have the following
inclusions
$$
T(\mathcal{O}(S)) \subset \pi(T(\mathcal{O}_v)) \subsetneqq
T(\mathcal{O}_v).
$$
This implies that the closure of $T(\mathcal{O}(S))$ in $\prod_{v
\notin S} T(\mathcal{O}_v)$ is of infinite index, and therefore the
closure of $T(K)$ in $T(\mathbb{A}_S)$ is of infinite index as well.

\vskip2mm

Next, this argument can be extended to an arbitrary torus $T$ over a
global field $K$ and any finite set $S$ of places of $K$. Moreover,
by considering coverings (isogenies) $\pi_m \colon T \to T$,
$\pi_m(t) = t^m$ for various $m$ prime to $\mathrm{char}\: K$, one
proves the following.
\begin{prop}\label{P:tori}
Let $T$ be a nontrivial torus over a global field $K$, and $S$ be a
finite set of places of $K$. If $\overline{T(K)}$ is the closure of
$T(K)$ in $T(\mathbb{A}_S)$ then the quotient
$$
T(\mathbb{A}_S) / \overline{T(K)}
$$
is a group of infinite exponent.
\end{prop}

This proposition yields a strong version of the fact that a
nontrivial torus over a global field always fails to have strong
approximation with respect to any finite set of places $S$.
Nevertheless, a torus may have strong approximation with respect to
some infinite (and co-infinite) sets $S$ - see Remark 3 after
Theorem \ref{SAT}.

\vskip2mm

\noindent {\bf 3. Simply connectedness as a necessary condition.}
The discussion of tori in the previous subsection suggests that the
existence of a nontrivial covering map for a given variety $X$ over
a global field $K$ may prevent it from having strong approximation
with respect to any finite set of places $S$. Indeed, as we will see
soon, simply connectedness of a connected  absolutely almost simple
group $G$ (i.e., the absence of nontrivial central isogenies $\pi
\colon \widetilde{G} \to G$ with connected $\widetilde{G}$ - see
\cite{Tits-cl} for a more detailed discussion) is one of the
essential conditions in the Strong Approximation Theorem for
algebraic groups (see Theorem \ref{SAT} below). But before we shift
our focus entirely to algebraic groups, we would like to mention the
following general result of Minchev \cite{Minch}  which does not
seem to be well-known to the general audience. (Note that we did not
formally define adeles for arbitrary varieties, so the reader may
want to assume that all varieties considered are actually affine, in
which case the definitions are completely  parallel to the above
definitions for algebraic groups.)
\begin{prop}\label{P:Minchev}
{\rm (\cite[Theorem 1]{Minch})} Let $X$ be an irreducible normal
variety over a number field $K$. If there exists a nontrivial
connected unramified covering $f \colon Y \to X$ defined over an
algebraic closure $\overline{K}$, then $X$ does not have strong
approximation with respect to any finite set $S$ of places of $K$.
\end{prop}

Since \cite{Minch} was published in a journal with rather limited
circulation (and only in Russian), we will give here a sketch of the
argument assuming $X$ and $Y$ to be affine and smooth and $S$ to
contain all archimedean places. We may assume that $f$ is a Galois
cover of degree $n > 1$, and pick a finite extension $L/K$ such that
$Y$ and $f$ are $L$-defined. For $x \in X(L)$, we let $L(f^{-1}(x))$
denote the extension of $L$ generated by the coordinates of all
preimages of $x$ in $Y(\overline{K})$; note that $[L(f^{-1}(x)) : L]
\leqslant n!$. Using the local version of the Chevalley-Weil theorem
(cf. \cite[Ch. 2, Lemma 8.3]{Lang-dioph}), for which we need $f$ to
be unramified, one shows that there exists a finite set of places
$S_1$ of $K$ containing $S$ such that any $v \notin S_1$ is
unramified in $L(f^{-1}(x))$ for all $x \in X(\mathcal{O}(S))$.
Invoking Hermite's theorem (cf. \cite[p. 122]{Lang-NT}), we now
conclude that there are only finitely many possibilities for
$L(f^{-1}(x))$ as $x$ ranges in $X(\mathcal{O}(S))$, and therefore
there exists a finite extension $L_1/L$ such that
$f^{-1}(X(\mathcal{O}(S))) \subset Y(L_1)$. Enlarging $L$, we can
actually assume that $L = L_1$ and $L/K$ is a Galois extension.
Also, expanding $S$ if necessary, we can make sure that if $v \notin
S$ splits completely in $L$ (i.e., $L \subset K_v$) then
\begin{equation}\label{E:Minchev0}
X(\mathcal{O}(S)) \subset f_{K_v}(Y(\mathcal{O}_v)).
\end{equation}
On the other hand, for almost all nonarchimedean places $w$ of $L$,
the reductions $\underline{X}^{(w)}$ and $\underline{Y}^{(w)}$
modulo $w$ are smooth irreducible varieties over the residue field
$\ell_w$, and the reduction $\underline{f}^{(w)} \colon
\underline{Y}^{(w)} \to \underline{X}^{(w)}$ is an $n$-sheeted
Galois cover. It follows that
\begin{equation}\label{E:Minchev1}
\vert \underline{f}_{\ell_w}^{(w)}(\underline{Y}^{(w)}(\ell_w))
\vert = \frac{\vert \underline{Y}^{(w)}(\ell_w) \vert}{n}.
\end{equation}
Since $\underline{X}^{(w)}$ and $\underline{Y}^{(w)}$ are
irreducible, by the Lang-Weil theorem \cite{LaWe}, the cardinalities
$\vert \underline{X}^{(w)}(\ell_w) \vert$ and $\vert
\underline{Y}^{(w)}(\ell_w) \vert$ are both ``approximately equal"
to $q_w^d$, where $q_w = \vert \ell_w \vert$ and $d$ is the common
dimension of $\underline{X}^{(w)}$ and $\underline{Y}^{(w)}$.
Comparing this with (\ref{E:Minchev1}), we see that for almost all
$w$, the cardinality $\vert
\underline{f}_{\ell_w}^{(w)}(\underline{Y}^{(w)}(\ell_w)) \vert$ is
only a fraction of $\vert \underline{X}^{(w)}(\ell_w) \vert$; in
particular,
$\underline{f}_{\ell_w}^{(w)}(\underline{Y}^{(w)}(\ell_w)) \neq
\underline{X}^{(w)}(\ell_w)$. Since by Hensel's lemma, the reduction
map $X(\mathcal{O}_w) \to \underline{X}^{(w)}(\ell_w)$ is
surjective, we obtain that
$$
f_{L_w}(Y(\mathcal{O}_w)) \neq X(\mathcal{O}_w).
$$
(in fact, our argument shows that $f_{L_w}(Y(\mathcal{O}_w))$ is
``much smaller" than -- in some sense, a ``fraction" of --
$X(\mathcal{O}_w)$).

This discussion, in conjunction with (\ref{E:Minchev0}) implies that
for almost all $v$ that  split completely in $L$, the set
$X(\mathcal{O}(S))$ is not dense in $X(\mathcal{O}_v)$. Since by
Tchebotarev's Density Theorem (\cite[p.~169]{Lang-NT}), there are
infinitely many $v$'s that split completely in $L$, we obtain that
$X$ does not have strong approximation with respect to $S$ (and in
fact that the closure of $X(\mathcal{O}(S))$ in $\prod_{v \notin S}
X(\mathcal{O}_v)$ is very ``thin"). \hfill $\Box$

\vskip2mm

(We note that Minchev \cite{Minch} points out another necessary
condition for strong approximation in a $K$-variety $X$ (which is
much easier to prove): $X$ needs to be (absolutely) irreducible.)

\vskip2mm

While the proof of Proposition \ref{P:Minchev} for general varieties
requires some facts from arithmetic algebraic geometry, there is a
much simpler argument in the case of algebraic groups (cf.
\cite[\S~7.4]{PlR}). Since most readers are likely to be
particularly interested in this case, we will explain the idea using
the following example. Consider the canonical isogeny
$$
\widetilde{G} = \mathrm{SL}_2 \stackrel{\pi}{\longrightarrow}
\mathrm{PSL}_2 = G
$$
of algebraic groups over a number field $K$. By the Skolem-Noether
theorem, one can think of $G$ as the automorphism group
$\mathrm{Aut}(M_2)$ of the degree two matrix algebra. Then for any
field extension $F/K$, again by the Skolem-Noether theorem, we have
$$
G(F) = \mathrm{Aut}_F(M_2(F)) = PGL_2(F).
$$
Then there is an exact sequence
\begin{equation}\label{E:es2}
\widetilde{G}(F) \stackrel{\pi_F}{\longrightarrow} G(F)
\stackrel{\theta_F}{\longrightarrow} F^{\times}/{F^{\times}}^2
\rightarrow 1,
\end{equation}
where $\theta_F$ is induced by the determinant, viz. $g F^{\times}
\mapsto (\det g) {F^{\times}}^2$. (Alternatively, one can think of
$G$ as the special orthogonal group $\mathrm{SO}_3(q)$ of the
Killing form $q$ on the Lie algebra $\mathfrak{sl}_2$ -- recall that
$q = 2x^2 + yz$ in the Chevalley basis; then $\widetilde{G}$ can be
identified with $\mathrm{Spin}_3(q)$, and $\theta_F$ becomes simply
the spinor norm map on $\mathrm{SO}_3(q)(F)$.)

The point is that given {\it  any} finitely generated subgroup
$\Gamma \subset G(K)$, its image $\Delta := \theta_K(\Gamma)$ is a
{\it finite group}. Now, if $K$ is a number field, it follows from
Tchebotarev's Density Theorem that there are infinitey many
nonarchimedean places $v$ of $K$ such that the image of $\Delta$
under the natural map $K^{\times}/{K^{\times}}^2 \rightarrow
K_v^{\times}/{K_v^{\times}}^2$ is trivial. From the exactness of
(\ref{E:es2}) for $F = K_v$, we conclude that for these $v$ we have
$$
\Gamma \subset \pi_{K_v}(\widetilde{G}(K_v)) \neq G(K_v).
$$
Applying this to $\Gamma = G(\mathcal{O}(S))$ (which is finitely
generated), we obtain that for almost all such $v$,
$$
G(\mathcal{O}(S)) \subset \pi_{K_v}(\widetilde{G}(\mathcal{O}_v))
\neq G(\mathcal{O}_v).
$$
The latter implies that the closure of $G(\mathcal{O}(S))$ in
$\prod_{v \notin S} G(\mathcal{O}_v)$ is of infinite index, for any
finite set $S$ of places of $K$, and hence $G$ fails to have strong
approximation.

\vskip2mm

This type of argument easily generalizes to prove that if  a
connected algebraic group $G$ over a number field $K$ is not simply
connected, then $G$ fails to have strong approximation for any
finite set $S$ of places of $K$ (see \cite[\S~7.4]{PlR} for the
details).

\vskip2mm

\noindent {\sc Example.} Let $G = \mathrm{GL}_2$. Set $\widetilde{G}
= G \times \mathbb{G}_m$. Then the product map $\widetilde{G} \to G$
is an isogeny of degree 2. Moreover, composing it with the map
$\widetilde{G} \to \widetilde{G}$, $(g , t) \mapsto (g , t^{\ell})$
for $\ell \geqslant 1$, we obtain an isogeny $\widetilde{G} \to G$
of an arbitrary even degree $2\ell$. On the other hand, the map $G
\to G$, $g \mapsto (\det g)^{\ell} g$ for $\ell \geqslant 1$, is an
isogeny of an arbitrary odd degree $(2\ell + 1)$. Thus, $G$ has
finite-sheeted connected coverings of any degree, in particular, it
is not simply connected. In view of the results discussed above,
this explains why $G$ does not have strong approximation with
respect to any finite $S$.

\vskip2mm

\noindent {\bf 4. Strong approximation theorem.} So far, we have
identified two necessary conditions for strong approximation in a
connected algebraic group $G$ over a number field $K$ with respect
to a finite set $S$ of places of $K$ that contains all archimedean
places: the $S$-arithmetic subgroups (i.e., subgroups commensurable
with $G(\mathcal{O}(S))$) must be Zariski-dense, and $G$ must be
simply connected. It turns out that for semi-simple groups, these
conditions are also sufficient. Since the general case easily
reduces to absolutely almost simple groups (cf. \cite[\S~7.4]{PlR}),
we will give a precise statement of the Strong Approximation Theorem
only for this case (however, we will include global fields of
positive characteristic).
\begin{thm}\label{SAT}
{\rm (Kneser \cite{Kneser}, Platonov \cite{Platonov} in
characteristic zero; Margulis \cite{Marg1}, \cite{Marg2}, Prasad
\cite{Prasad} in positive characteristic)} Let $G$ be a connected
absolutely almost simple algebraic group over a global field $K$,
and let $S$ be a finite set of places of $K$. Then $G$ has strong
approximation with respect to $S$ (i.e., $G(K)$ is dense in
$G(\mathbb{A}_S)$) if and only if

\vskip2mm

{\rm (1)} $G$ is simply connected;

\vskip1mm

{\rm (2)} $G_S := \prod_{v \in S} G(K_v)$ is noncompact.
\end{thm}

(We note that for an absolutely almost simple group $G$, condition
(2) is equivalent to $G(\mathcal{O}(S))$ being infinite, and hence
Zariski-dense in $G$, cf. \cite[Theorem 4.10]{PlR}. It should also
be mentioned that in the statement of the theorem we included only
the names of the main contributors; the interested reader will find
more historical remarks at the end of \S~7.4 in \cite{PlR}, and also
at the end of the current section.)

\vskip2mm

{\bf Remarks.} 1. The condition that $G$ is simply connected is used
in the proof of sufficiency in Theorem \ref{SAT} in a very peculiar
way that is totally unrelated to the above considerations showing
that simply connectedness is necessary for strong approximation.
More precisely, what we need is the fact that for all  $v \notin S$
such that $G$ is $K_v$-isotropic (i.e., has positive rank over
$K_v$), the group $G(K_v)$ does not have proper (abstract) subgroups
of finite index (see \S~2.6). It turns out that in the situation at
hand, for $G$ simply connected, the group $G(K_v)$ does not, in
fact, have {\it any} proper noncentral normal subgroups. To put this
result in perspective, we recall the result of Tits \cite{Ti}
asserting that given an absolutely almost simple isotropic algebraic
group $G$ over a field $P$ with $\geqslant 4$ elements, the subgroup
$G(P)^+$ of $G(P)$ generated by the $P$-rational points of
$P$-defined parabolics, does not have any proper noncentral normal
subgroups. In the same paper, Tits proposed a conjecture, which
later became known as the {\it Kneser-Tits conjecture}, that
actually $G(P)^+ = G(P)$ {\it if} $G$ is simply connected. While
over general fields this conjecture turned out to be false (cf.
Platonov \cite{Plat-KT}), it was proved by Platonov \cite{Platonov}
to hold over nonarchimedean local fields of characteristic zero
(i.e., finite extensions of the $p$-adic field $\Q_p$); over $\R$
this fact was established much earlier by E.~Cartan (cf.
\cite[Proposition~7.6]{PlR}). This connection between strong
approximation and the Kneser-Tits conjecture was the centerpiece of
Platonov's paper \cite{Platonov}. We will see another manifestation
of this connection in the analysis of strong approximation for
arbitrary Zariski-dense subgroups (cf. \S~\ref{S:ZD}), although in a
different setting (viz., over finite fields). On the other hand,
over a local or a finite field $P$, we have $G(P)^+ \neq G(P)$ if
$G$ is not simply connected, and hence in this case $G(P)$ does have
proper noncentral normal subgroups (of finite index). This is where
the proof of Theorem \ref{SAT} and the corresponding argument in \S
\ref{S:ZD} breaks down if one drops the assumption that $G$ is
simply connected. Finally, we remark that the Kneser-Tits conjecture
has generated a lot of research not associated with strong
approximation - see Gille \cite{Gille} for a recent survey.

\vskip2mm

2. The effect of non-simply connectedness on strong approximation
with respect to a finite set $S$ is different for tori and
semi-simple groups: for a $K$-torus $T$, the quotient
$T(\mathbb{A}_S)/\overline{T(K)}$ by the closure of the group of
rational points has infinite exponent (Proposition \ref{P:tori}),
while, as follows from Theorem \ref{SAT}, for a connected absolutely
almost simple non-simply connected $K$-group $G$ with a universal
$K$-defined cover $\pi \colon \widetilde{G} \to G$ such that the
group $\widetilde{G}_S$ is not compact, the closure $\overline{G(K)}
\subset G(\mathbb{A}_S)$ is a normal subgroup with the infinite
quotient $G(\mathbb{A}_S)/\overline{G(K)}$ having finite exponent.
(This distinction, of course, reflects the fact that the (algebraic)
fundamental group of $G$ is finite, while that of $T$ is infinite.)

\vskip2mm

3. A connected $K$-group $G$ may  have strong approximation with
respect to certain {\it infinite} sets $S$ of places of $K$ {\it
without} being simply connected. For example, in \cite{PrRap}, we
examined in this context strong approximation in tori (which can
never be valid for finite $S$ - see Proposition \ref{P:tori}). To
avoid technical definitions, we will just indicate what our results
give in the case of the 1-dimensional split torus $T = \mathbb{G}_m$
over $K = \Q$: {\it If $S$ is an infinite set of places of $K$ that
contains the $p$-adic places for almost all primes $p$ in a certain
arithmetic progression, then the closure $\overline{T(\Q)}$ of
$T(\Q)$ in the group of $S$-adeles $T(\mathbb{A}_S)$ is of finite
index}. The result for general tori is basically the same but
contains one important exclusion that has to do with how the
arithmetic progression interacts with the splitting field of the
torus. This fact is instrumental for the analysis of the congruence
subgroup problem: it implies, in particular, that if $G$ is an
absolutely almost simple simply connected  algebraic group over a
number field $K$, which is an inner form, and $S$ is a set of places
of $K$ that contains all archimedean places and also almost all
places in a certain generalized arithmetic progression, then the
corresponding congruence kernel $C^S(G)$ is trivial, i.e. every
subgroup of finite index in $G(\mathcal{O}(S))$ contains a suitable
congruence subgroup (provided that $G(K)$ has a standard description
of normal subgroups), see \cite{PrRap-CSP}.

\vskip2mm

4. For general affine varieties, the analogs of conditions (1) and
(2) in Theorem \ref{SAT} may not be sufficient for strong
approximation, even in homogeneous spaces.

\vskip1.5mm

\noindent {\sc Example.} Let $f(x , y , z) = ax^2 + by^2 + cz^2$ be
the nondegenerate ternary quadratic form over a number field $K$,
and let $X \subset \mathbb{A}^3$ be a quadric given by $f(x , y , z)
= a$. Set $g(x , y) = by^2 + cz^2$. Let $S$ be a finite set of
places of $K$ such that $X_S = \prod_{v \in S} X(K_v)$ is noncompact
(equivalently, for some $v \in S$ the form $f$ is $K_v$-isotropic).
Then (see \S~2.5 below)
$X$ has strong approximation with respect to $S$ if and only if one
of the following two conditions holds:

\vskip1mm

\noindent (a) $g$ is $K$-isotropic;

\vskip.6mm

\noindent (b) \parbox[t]{11.9cm}{$g$ is $K$-anisotropic and there
exists $v \in S$ such that $g$ remains anisotropic over $K_v$ and
either $v$ is nonarchimedean or $f$ is $K_v$-isotropic.}

\vskip2mm

\noindent It follows that a rational quadric $X$ defined by $x_1^2 +
x_2^2 - 2x_3^2 = 1$ (which is simply connected) does not have strong
approximation with respect to $S = \{ \infty \}$.

\vskip3mm

\noindent {\bf 5. Strong approximation in homogeneous spaces.} The
fact quoted in the above example is a consequence of the analysis of
strong approximation in (affine) homogeneous spaces of algebraic
groups. Since these results (cf. \cite{Borovoi}, \cite{Rap}; a
detailed exposition of \cite{Rap} was given in \cite{Rap-CSP}) are
not as widely known as Theorem \ref{SAT}, we briefly mention some of
them here for the sake of completeness. The fact that only connected
simply connected varieties have a chance to possess strong
approximation, by and large, forces us to focus our attention of
homogeneous spaces of the form $X = G/H$ where $G$ is a semi-simple
simply connected algebraic $K$-group, and $H$ is a $K$-defined
connected reductive subgroup (any such variety is affine and simply
connected). Furthermore, given a set $S$ of places of $K$, it is not
difficult to show that for such $X$, the space $X_S$ is noncompact
if and only if $G_S$ is noncompact. Assuming now that $G$ is
actually absolutely almost simple, we conclude from Theorem
\ref{SAT} that $G$ has strong approximation with respect to $S$ (for
a general semi-simple group $G$ we need to consider its simple
components). Then using Galois cohomology one investigates when
strong approximation in $G$ implies strong approximation in $X =
G/H$. Here is one easy result in this direction.
\begin{prop}\label{P:HS1}
{\rm (\cite{Rap})} Let $X = G/H$ be the quotient of a connected
absolutely almost simple simply connected algebraic group $G$
defined over a number field $K$ by a connected semi-simple simply
connected $K$-subgroup $H$. Then $X$ has strong approximation with
respect to a finite set $S$ of places of $K$ if and only the space
$X_S = \prod_{v \in S} X(K_v)$ is noncompact.
\end{prop}

Now, let $q = q(x_1, \ldots , x_n)$ be a nondegenerate quadratic
form in $n \geqslant 3$ variables. Consider the quadric $X \subset
\mathbb{A}^n$ given by the equation $q(x_1, \ldots , x_n) = a$ for
some $a \in K^{\times}$. Assuming that $X(K) \neq \emptyset$, pick
$x \in X(K)$. Then $X = G/H$ where $G = \mathrm{Spin}_n(q)$ and $H =
G(x)$ (the stabilizer of $x$); note that $H \simeq
\mathrm{Spin}_{n-1}(q')$ where $q'$ is the restriction of $q$ to the
orthogonal complement of $x$. So, it follows from Proposition
\ref{P:HS1} that for $n \geqslant 5$, the quadric $X$ has strong
approximation with respect to $X$ if and only if there exists $v \in
S$ such that $q$ is $K_v$-isotropic. The same result remains valid
for $n = 4$ even though in this case $G$ is not absolutely almost
simple. (Incidentally, this result applies to the equation that
defines $\mathrm{SL}_2$ (cf. \S~1.2), yielding thereby another proof
of strong approximations for this group, cf. Lemma \ref{L:2}.)

The case $n = 3$ is different as here $H$ is a torus. This case can
also be treated in a rather explicit form using the results of
Nakayama-Tate on the Galois cohomology of tori. More precisely, let
$T$ be a $K$-torus, and let $L$ be the splitting field of $T$. As
usual, given a module $M$ over the Galois group $\Ga(L/K)$, we let
$H^i(L/K , M)$ denote the Galois cohomology group $H^i(\Ga(L/K) ,
M)$. Given a finite set $S$ of places of $K$, we let $\overline{S}$
denote the set of all extensions of places in $S$ to $L$, and let
$\mathbb{A}_L$ and $\mathbb{A}_{L , \overline{S}}$ denote the rings
of adeles and $\overline{S}$-adeles of $L$. Finally, let $c_L(T) =
T(\mathbb{A}_L)/T(L)$ be the adele class group of $T$ over $L$, and
let
$$
\delta \colon H^1(L/K , T(\mathbb{A}_L)) \longrightarrow H^1(L/K ,
c_L(T))
$$
be the corresponding map on cohomology. Then, viewing
$T_{\overline{S}}$ and $T(\mathbb{A}_{L , \overline{S}})$ as
subgroups of $T(\mathbb{A}_L)$, we have the following statement.
\begin{prop}\label{P:HS2}
{\rm (\cite{Rap})} Let $X = G/T$, where $G$ is an absolutely almost
simple simply connected $K$-group and $T \subset G$ is a $K$-torus.
Then $X$ has strong approximation with respect to a finite set $S$
of places of $K$ if and only if $X_S$ is noncompact and
$$
\delta\left( H^1(L/K , T(\mathbb{A}_{L , \overline{S}}))\right)
\subset \delta\left(\mathrm{Ker}\left(H^1(L/K , T_{\overline{S}})
\to H^1(L/K , G_{\overline{S}})\right) \right),
$$
where $L$ is the splitting field of $T$ and $\overline{S}$ consists
of all extensions of places in $S$ to $L$.
\end{prop}

This proposition yields the criterion for strong approximation for
the quadrics defined by ternary forms we used in Remark 4 of \S~2.4.
It also implies that for $X = G/T$, one can find a finite set of
places $S_0$ (depending on $T$) such that $X$ has strong
approximation with respect to $S$ whenever $S \supset S_0$. It turns
out that this qualitative statement remains valid for quotients by
arbitrary connected reductive subgroups. More precisely, using some
ideas that eventually led him to theorems of the Nakayama-Tate type
for Galois cohomology of arbitrary connected groups, Borovoi proved
the following.
\begin{prop}\label{P:HS3}
{\rm (\cite{Borovoi})} Let $X = G/H$ be the quotient of a connected
absolutely almost  simple algebraic group $G$ over a number field
$K$ by its connected reductive $K$-defined subgroup $H$. There
exists a finite set $S_0$ of places of $K$ such that $X$ has strong
approximation with respect to $S_0$ (and then, of course, it also
has strong approximation with respect to any $S \supset S_0$).
\end{prop}

We remark in passing that the results on strong approximation in
homogeneous spaces were used to extend Kneser's method for proving
the centrality of the congruence kernel for spinor groups to groups
of other classical types as well as $G_2$ \cite{Rap}, \cite{Rap1},
\cite{Rap-CSP} (cf. also \cite{Tom1}, \cite{Tom2}), to establish
bounded generation of some $S$-arithmetic subgroups in orthogonal
groups \cite{ER}, and to study some Diophantine questions involving
quadratic forms \cite{CX} (we should mention that the results of the
latter work were recently generalized in \cite{Demarche} where the
deviation from strong approximation in a connected $K$-group $G$ has
been expressed in terms of a certain subquotient of the Brauer group
of $G$).

\vskip3mm

\noindent {\bf 6. On the proof of sufficiency in Theorem \ref{SAT}.}
We begin with the following statement that applies to arbitrary
Zariski-dense subgroups.
\begin{lemma}\label{L:SA}
Let $G$ be an absolutely almost simple algebraic $\Q$-group, and let
$\Gamma \subset G(\Z)$ be a Zariski-dense subgroup of $G$. Then for
any prime $p$, the closure $\overline{\Gamma}^{(p)} \subset G(\Z_p)$
is open.
\end{lemma}
\begin{proof}
Let $\mathfrak{g}$ be the Lie algebra of $G$ as an algebraic group,
so that $\mathfrak{g}_{\Q_p}$ is the Lie algebra of $G(\Z_p)$ as a
$p$-adic analytic group. By a theorem of Cartan (cf.
\cite[Theorem~3.4]{PlR}), $\Delta := \overline{\Gamma}^{(p)}$ is a
$p$-adic Lie group, of positive dimension as $\Gamma$ is
non-discrete in $G(\Z_p)$ (the discreteness would force it to be
finite, and therefore prevent it from being Zariski-dense). So, the
Lie algebra $\mathfrak{h}$ of $\Delta$ as a $p$-adic analytic group
is a nonzero $\Q_p$-subalgebra of $\mathfrak{g}_{\Q_p}$. Clearly,
$\mathfrak{h}$ is invariant under $\mathrm{Ad}\: \Gamma$, so the
Zariski-density of $\Gamma$ in $G$ implies that $\mathfrak{h}
\otimes_{\Q_p} \overline{\Q_p}$ is invariant under $\mathrm{Ad}\:
G$. Since the adjoint representation of $G$ on $\mathfrak{g}$ is
irreducible, we conclude that $\mathfrak{h} = \mathfrak{g}_{\Q_p}$,
and therefore $\Delta$ is open in $G(\Z_p)$ by the Implicit Function
Theorem.
\end{proof}

As we will discuss at the beginning of \S~\ref{S:ZD}, this lemma,
though useful,  falls short of proving any definite form of strong
approximation. We will now indicate additional considerations needed
to prove the sufficiency in Theorem \ref{SAT} in characteristic
zero, following Platonov's original argument \cite{Platonov}.  Let
us assume that $S$ contains all archimedean valuations of $K$. In
this case, it is easy to see from the definition of the topology on
$G(\mathbb{A}_S)$ that strong approximation is equivalent to the
following statement:

\vskip2mm

\begin{center}

{\it for any finite set of places $S_1$ of $K$ which is disjoint
from $S$,

the group $G(\mathcal{O}(S \cup S_1))$ is dense in $G_{S_1} :=
\prod_{v \in S_1} G(K_{v})$.}

\end{center}

\vskip2mm

\noindent To showcase the idea, we will now prove this statement in
the case where $K = \Q$ and $S_1 = \{ p \}$, a single $p$-adic place
such that $G$ is $\Q_p$-isotropic - see \cite[\S 7.4]{PlR} for the
general case. First, by the reduction theory for $S$-arithmetic
groups, $G(\mathcal{O}(S \cup S_1))$ is a {\it lattice} (i.e., a
discrete subgroup of finite covolume) in $G_{S \cup S_1}$, see
\cite[Theorem~5.7]{PlR}. Since by assumption the group $G_S$ is
non-compact, it is not difficult to show (cf. \cite[Lemma
3.17]{PlR}) that $G(\mathcal{O}(S \cup S_1)) \subset G(\Q_p)$ is
nondiscrete (in particular, infinite), and if $\Delta$ denotes the
$p$-adic closure $\overline{G(\mathcal{O}(S \cup S_1))}^{(p)}$ then
$G(\Q_p)/\Delta$ carries a finite invariant measure. Next, the fact
that $G(\mathcal{O}(S \cup S_1))$ is infinite implies that it is
actually Zariski-dense in $G$ (Borel's Density Theorem, cf.
\cite[Theorem 4.10]{PlR}). Taking into account the nondiscreteness
of $G(\mathcal{O}(S \cup S_1))$ in $G(\Q_p)$ and repeating the proof
of Lemma \ref{L:SA}, we conclude that $\Delta$ is open in $G(\Q_p)$.
Then the existence of a finite invariant measure on $G(\Q_p)/\Delta$
implies that $\Delta \subset G(\Q_p)$ is a subgroup of finite index.
On the other hand, since the group $G$ is connected, absolutely
almost simple, simply connected and $\Q_p$-isotropic, by the
Kneser-Tits conjecture over $p$-adic fields we have $G(\Q_p) =
G(\Q_p)^+$, and therefore the group $G(\Q_p)$ does not have any
proper noncentral normal subgroup. In particular, it does not
contain any proper subgroups of finite index, and we obtain that
$\Delta = G(\Q_p)$, as required.

\vskip2mm

This argument breaks down in positive characteristic, first and
foremost, because Cartan's theorem, which is at the heart of the
proof of Lemma \ref{L:SA}, is valid only in characteristic zero. It
should be mentioned that eventually Pink \cite{Pink1} proved a
result which in some sense can be viewed as an analog (or
replacement) of Cartan's theorem. The precise general statement is
too technical for us to discuss here, so we will only indicate what
it yields in one particular case (see Theorem 0.7 in \cite{Pink1}):
{\it Let $G$ be an absolutely simple connected adjoint group over a
local field $F$, and assume that the adjoint representation of $G$
is irreducible. If $\Gamma \subset G(F)$ is a compact Zariski-dense
subgroup, then there exists a closed subfield $E \subset F$ and a
model $H$ of $G$ over $E$ such that $\Gamma$ is open in $H(E)$.}
This sort of result can be used to prove Theorem \ref{SAT} in
positive characteristic,  but the original argument given virtually
simultaneously by Margulis \cite{Marg1} and Prasad \cite{Prasad},
was different. They derived strong approximation (arguing along the
lines indicated above) from the following statement:

\vskip2mm

\begin{center}

{\it Let $G$ be a connected semi-simple algebraic group over a local
field $F$, and

\vskip.2mm

let $H \subset G(F)$ be a nondiscrete closed subgroup such that
$G(F)/H$

\vskip.2mm

carries a finite invariant Borel measure. Then $H \supset G(F)^+$.}

\end{center}

\vskip2mm

\noindent Their argument used ergodic considerations and
representation theory. More than 25 years later, Pink \cite{Pink3}
used his results from \cite{Pink1} to give a purely algebraic proof
of this theorem, and hence of strong approximation.

\section{Strong approximation for Zariski-dense
subgroups}\label{S:ZD}

\noindent {\bf 1. Overview.} The Strong Approximation Theorem
\ref{SAT} gives us {\it precise} information about the adelic
closure of $S$-arithmetic subgroups: for example,  if $G$ is an
algebraic $\Q$-group that has strong approximation with respect to
$S = \{ \infty \}$ then for any matrix realization of $G$, the group
$G(\Z)$ is dense in $G(\hat{\Z}) = \prod_{p} G(\Z_p)$ - cf. \S
\ref{S:AG}. At the same time, as we explained in \S \ref{S:I}, one
can expect a general {\it qualitative} openness result for the
adelic closure of an arbitrary Zariski-dense subgroup. The goal of
this section is to discuss some results in this direction.

\vskip2mm

First, we note one consequence of Lemma \ref{L:SA}. Let $G$ be a
connected absolutely almost simple algebraic $\Q$-group, and let
$\Gamma \subset G(\Z)$ be a Zariski-dense subgroup of $G$. Then
using the fact that $G(\Z_p)$ is a virtually pro-$p$ group, one
easily deduces from Lemma \ref{L:SA} that given a {\it finite} set
$S$ of distinct primes, the closure
$$
\overline{\Gamma}^{(S)} \subset \prod_{p \in S} G(\Z_p)
$$
is open. This  statement is already sufficient for some
applications; for example, it was used in \cite{PrRap-reg} to prove
the existence of generic elements in arbitrary finitely generated
Zariski-dense subgroups $\Gamma \subset G(K)$,  where $G$ is a
semi-simple algebraic group over a finitely generated field $K$ of
characteristic zero; see \cite{GorN}, \cite{JKZ} and \cite{LuRosen}
for more recent work in this direction. (In his talk at the
workshop, G.~Prasad surveyed applications of generic elements to the
analysis of isospectral locally symmetric spaces, cf.
\cite{PrRap-Iso}, \cite{PrRap-NT}.) On the other hand, if we take
$S$ to be the set of {\it all} primes, the best we can get from
Lemma \ref{L:SA} is the following:

\vskip2mm

\hskip.6cm \parbox[t]{10cm}{{\it the closure $\widehat{\Gamma}$ of
$\Gamma$ in $G(\hat{\Z}) = \prod_p G(\Z_p)$ contains $\prod_p W_p$
\centerline{where $W_p \subset G(\Z_p)$ is open for each $p$.}}}
\hfill $(*)$

\vskip2mm

\noindent Of course, this does {\bf not} imply that
$\widehat{\Gamma}$ is open in $G(\hat{\Z})$ - for this we need to
show that actually $W_p = G(\Z_p)$ for almost all $p$. The first
general result in this direction was the following.
\begin{thm}\label{T:MVW}
{\rm (Matthews, Vaserstein, Weisfeiler \cite{MVW})} Let $G$ be a
connected absolutely almost simple simply connected algebraic group
over $\Q$.

\vskip2mm

\noindent {\rm (1)} \parbox[t]{12cm}{If $\Gamma \subset G(\Z)$ is a
Zariski-dense subgroup, then the closure $\widehat{\Gamma} \subset
G(\hat{\Z})$ is open.}

\vskip2mm

\noindent {\rm (2)} \parbox[t]{12cm}{If $\Gamma \subset G(\Q)$ is a
finitely generated Zariski-dense subgroup,  then for some finite set
$S$ of places of $\Q$ containing $\infty$, the closure of $\Gamma$
in the group of $S$-adeles $G(\mathbb{A}_S)$ is open.}
\end{thm}

The paper \cite{MVW} appeared in 1984, but the interest in these
sorts of results arose at least 20 years earlier in connection with
the study of Galois representations on torsion points of elliptic
curves. In fact, in his book \cite{Serre} on $\ell$-adic
representations, Serre pretty much had this theorem for $G =
\mathrm{SL}_2$ (at least, all the ingredients of the proof were
there).

\vskip2mm

Parts (1) and (2) are proved in the same way,  so let us focus our
discussion on the proof of (1) as this will allow us to keep our
notations simple. First, it enough to prove that for almost all
primes $p$, the closure $\overline{\Gamma}^{(p)} \subset G(\Z_p)$
coincides with $G(\Z_p)$. This reduction step is achieved using
$(*)$ in conjunction with the fact that for almost all primes $p$,
the group $G$ has a smooth reduction $\Gu^{(p)}$ modulo $p$ and the
groups $\Gu^{(p)}(\mathbb{F}_p)$ are pairwise non-isomorphic almost
simple groups (for the reader who is interested only in the case $G
= \mathrm{SL}_n$, we will indicate that here, of course, $\Gu^{(p)}
= \mathrm{SL}_n / \mathbb{F}_p$, and the structural facts quoted
above are well-known). Next, it turns out that for almost all $p$,
proving that $\overline{\Gamma}^{(p)} = G(\Z_p)$ reduces to showing
that  the reduction map $\rho_p \colon G(\Z_p) \to
\Gu^{(p)}(\mathbb{F}_p)$ has the property $\rho_p(\Gamma) =
\Gu^{(p)}(\mathbb{F}_p)$.
\begin{prop}\label{P:MVW}
{\rm (cf. \cite[7.3]{MVW})}For almost all $p$, if $\Delta \subset
G(\Z_p)$ is a closed subgroup such that $\rho_p(\Delta) =
\Gu^{(p)}(\mathbb{F}_p)$ then $\Delta = G(\Z_p)$.
\end{prop}

The proof for $G = \mathrm{SL}_2$ was given by Serre \cite[Ch. IV,
3.4]{Serre}.
\begin{lemma}\label{L:SL2}
Let $\Delta \subset SL_2(\mathbb{Z}_p)$, where $p > 3$, be a closed
subgroup such that for the reduction map $\rho_p \colon SL_2(\Z_p)
\to SL_2(\mathbb{F}_p)$ we have $\rho_p(\Delta) =
SL_2(\mathbb{F}_p)$. Then $\Delta = SL_2(\Z_p)$.
\end{lemma}
\begin{proof}
By assumption, there exists $g \in \Delta$ such that
$$
g = \left( \begin{array}{cc} 1 & 1 \\ 0 & 1 \end{array}  \right) +
ps \ \ \text{with} \ \ s \in M_2(\Z_p).
$$
We claim that
\begin{equation}\label{E:p-power}
g^p = \left( \begin{array}{cc} 1 & p \\ 0 & 1 \end{array}  \right) +
p^2t \ \ \text{with} \ \ t \in M_2(\Z_p).
\end{equation}
Indeed,
$$
g^p = \left(I_2 + \left( \left( \begin{array}{cc} 0 & 1 \\ 0 & 0
\end{array}  \right) + ps\right) \right)^p =
$$
$$
I_2 + p \left( \left( \begin{array}{cc} 0 & 1 \\ 0 & 0
\end{array}  \right) + ps \right) + \binom{p}{2} \left( \left( \begin{array}{cc} 0 & 1 \\ 0 & 0
\end{array}  \right) + ps\right)^2 + \cdots + \left( \left( \begin{array}{cc} 0 & 1 \\ 0 & 0
\end{array}  \right) + ps\right)^p.
$$
But clearly
$$
\left( \left( \begin{array}{cc} 0 & 1 \\ 0 & 0
\end{array}  \right) + ps\right)^k \equiv O_2 (\mathrm{mod}\: p) \ \
\text{for any} \ \ k \geqslant 2,
$$
and in fact
$$
\left( \left( \begin{array}{cc} 0 & 1 \\ 0 & 0
\end{array}  \right) + ps\right)^k \equiv O_2 (\mathrm{mod}\: p^2) \
\ \text{for any} \ \ k \geqslant 4
$$
as $\displaystyle \left( \begin{array}{cc} 0 & 1 \\ 0 & 0
\end{array} \right)^2 = O_2$ (the zero matrix). So, since $p > 3$,
the equation (\ref{E:p-power}) follows.

As $g^p \in \Delta$, we conclude from (\ref{E:p-power}) that the
image $\Phi$ of the intersection $\Delta \cap SL_2(\Z_p , p)$ with
the congruence subgroup modulo $p$ in
$$
SL_2(\Z_p , p) / SL_2(\Z_p , p^2) \simeq
\mathfrak{sl}_2(\mathbb{F}_p),
$$
where $\mathfrak{sl}_2$ is the Lie algebra of $\mathrm{SL}_2$ (i.e.,
$2 \times 2$-matrices with trace zero), is {\it nontrivial}. On the
other hand, $\Phi$ is obviously invariant under $\Delta$, and as
$\rho_p(\Delta) = SL_2(\mathbb{F}_p)$, it is actually invariant
under $SL_2(\mathbb{F}_p)$. But since $p \neq 2$, the group
$SL_2(\mathbb{F}_p)$ acts on $\mathfrak{sl}_2(\mathbb{F}_p)$
irreducibly, implying that $\Delta \cap SL_2(\Z_p , p)$ surjects
onto $SL_2(\Z_p , p)/SL_2(\Z_p , p^2)$. However, $SL_2(\Z_p , p)$ is
in fact the Frattini subgroup of the pro-$p$ group $SL_2(\Z_p , p)$,
so the latter fact implies that $\Delta \cap SL_2(\Z_p , p) =
SL_2(\Z_p , p)$, and our claim follows.
\end{proof}

The general case in Proposition \ref{P:MVW} is obtained by reduction
to the case of $\mathrm{SL}_2$. For this one observes that  the
group $G$ is quasi-split, and therefore $G(\Z_p)$ contains $H =
SL_2(\Z_p)$, for almost all $p$. We refer the reader to \cite{MVW}
for further details. (Note that one needs to argue a bit more
carefully on p. 529 in \cite{MVW} to make sure that $\Delta \cap H$
maps onto $SL_2(\mathbb{F}_p)$ surjectively; this can be achieved by
choosing a special $H$.)

\vskip1mm

So, to complete the proof of (both parts of) Theorem \ref{T:MVW},
one needs to prove the following.
\begin{thm}\label{T:Red}
Let $G$ be a connected absolutely almost simple simply connected
algebraic group over $\Q$, and let $\Gamma \subset G(\Q)$ be a
finitely generated Zariski-dense subgroup. Then there exists a
finite set of primes  $\Pi = \{p_1, \ldots , p_r\}$ such that

\vskip2mm

\noindent {\rm (1)} $\displaystyle \Gamma \subset G(\Z_{\Pi})$ where
$\Z_{\Pi} = \Z\left[ \frac{1}{p_1}, \ldots , \frac{1}{p_r} \right]$;

\vskip1mm

\noindent {\rm (2)} for $p \notin \Pi$ there exists a smooth
reduction $\Gu^{(p)}$;

\vskip1mm

\noindent {\rm (3)} \parbox[t]{12cm}{if $p \notin \Pi$ and $\rho_p
\colon G(\Z_p) \to \Gu^{(p)}(\mathbb{F}_p)$ is the corresponding
reduction map then $\rho_p(\Gamma) = \Gu^{(p)}(\mathbb{F}_p)$.}
\end{thm}

The conditions (1) and (2) are routine (in fact, (1) holds
automatically if $\Gamma \subset G(\Z)$), so the main point is to
ensure condition (3). The general idea is the following. Let
$\mathfrak{g}$ and $\underline{\mathfrak{g}}^{(p)}$ be the Lie
algebras of $G$ and $\Gu^{(p)}$. Since $\Gamma$ is Zariski-dense in
$G$, we conclude that $\mathrm{Ad}\: \Gamma$ acts on
$\mathfrak{g}_{\Q}$ absolutely irreducibly. By Burnside's Theorem
this means that $\mathrm{Ad}\: \Gamma$ spans $\mathrm{End}_{\Q}\:
\mathfrak{g}_{\Q}$ as a $\Q$-vector space. Excluding finitely many
primes, we can achieve that for any of the remaining primes $p$, the
group $\mathrm{Ad}\: \rho_p(\Gamma)$ acts on
$\underline{\mathfrak{g}}^{(p)}_{\mathbb{F}_p}$ absolutely
irreducibly. This eventually implies that for almost all $p$ we have
$\rho_p(\Gamma) = \Gu^{(p)}(\mathbb{F}_p)$. This implication would
be obvious if we could say that say that  $\rho_p(\Gamma)$ is
necessarily of the form $H(\mathbb{F}_p)$, where $H \subset
\Gu^{(p)}$ is some connected algebraic $\mathbb{F}_p$-subgroup.
(Indeed, then the Lie algebra $\mathfrak{h}$ of $H$ would be a
nonzero $\rho_p(\Gamma)$-invariant subspace of
$\underline{\mathfrak{g}}^{(p)}$, so $\mathfrak{h} =
\underline{\mathfrak{g}}^{(p)}$ and $H = \Gu^{(p)}$, as $\Gu^{(p)}$
is connected for almost all $p$, yielding the required fact.) Of
course, such an a priori description of $\rho_p(\Gamma)$ would be
too much to hope for, but important information along these lines,
which is sufficient for the proof of Theorem \ref{T:Red}, is
contained in a theorem due to Nori \cite{Nori}

\vskip3mm

\noindent {\bf 2. Theorem of Nori.} Let $H$ be an {\it arbitrary
subgroup} of $GL_n(\mathbb{F}_p)$. Set
$$
X = \{ x \in H \: \vert \: x^p = 1 \}
$$
(we will write $1$ in place of $I_n$ to simplify notations). Note
that if we assume that $p > n$ (which we will throughout this
subsection), then the condition $x^p = 1$ characterizes precisely
unipotent elements, i.e. is equivalent to the condition $(x - 1)^n =
0$. For $x \in X$, we can define
$$
\log x := - \sum_{i = 1}^{p-1} \frac{(1 - x)^i}{i}.
$$
Furthermore, observing that $(\log x)^n = 0$, we see that for any $t
\in \overline{\mathbb{F}_p}$ (algebraic closure), we can define
$$
x(t) := \exp(t \cdot \log x) \ \ \text{where} \ \ \exp z = \sum_{i =
0}^{p-1} \frac{z^i}{i!}
$$
(note that $x(1) = x$). We will regard $x(t)$ as a 1-parameter
subgroup $\mathbb{G}_a \to \mathrm{GL}_n$. Set
$$
H^+ = \left\langle X \right\rangle \subset H,
$$
and let $\widetilde{H}$ denote the connected $\mathbb{F}_p$-subgroup
of $\mathrm{GL}_n$ generated by the 1-parameter subgroups $x(t)$ for
all $x \in X$.
\begin{thm}\label{T:Nori}
{\rm (\cite{Nori})} If $p$ is large enough (for a given $n$), then
$H^+$ coincides with $\widetilde{H}(\mathbb{F}_p)^+$, the subgroup
of $\widetilde{H}(\mathbb{F}_p)$ generated by all unipotents
contained in it.
\end{thm}

Thus, Nori's theorem asserts that if $p$ is large enough compared to
$n$, then any subgroup of $GL_n(\mathbb{F}_p)$ generated by
$p$-elements  is essentially the group of $\mathbb{F}_p$-points of
some connected $\mathbb{F}_p$-defined algebraic subgroup of
$\mathrm{GL}_n$. Actually, in his paper \cite{Nori}, Nori proves a
stronger result stating that for a field $F$ which either has
characteristic zero or positive characteristic $p$ that is large
enough compared to $n$, the maps $\log$ and $\exp$ yield bijective
correspondences between nilpotently generated Lie subalgebras of
$M_n(F)$ and exponentially generated subgroups of $GL_n(F)$ (we
refer the reader to the original paper \cite{Nori} for precise
definitions and detailed statements of these results). The argument
in \cite{Nori} was based on algebro-geometric ideas; a different
proof was given by Hrushovski and Pillay \cite{HrPi} using
model-theoretic techniques (the idea of their argument is explained
in \cite[pp. 399-400]{LubSeg}). A vast generalization of Nori's
theorem is contained in a recent paper of Larsen and Pink
\cite{LarPi} describing the structure of finite linear groups over
fields of positive characteristic.

\vskip2mm

Given the nature of this article, we will not be able to discuss any
details of  Nori's argument. All we can offer as  compensation is a
proof of Nori's results for $GL_2(\mathbb{F}_p)$.
\begin{lemma}\label{L:SL21}
Let $H \subset GL_2(\mathbb{F}_p)$ be a subgroup of order divisible
by $p$, and let $H_p \subset H$ be a Sylow $p$-subgroup. Then either
$H_p \vartriangleleft H$ or $H \supset SL_2(\mathbb{F}_p)$.
\end{lemma}
\begin{proof}
We may assume that $H_p$ coincides with
$$
U := \left\{ \: \left( \begin{array}{cc} 1 & a \\ 0 & 1 \end{array}
\right) \ \vert \ a \in \mathbb{F}_p \: \right\}.
$$
It is well-known that the normalizer of $U$ in $GL_2(\mathbb{F}_p)$
coincides with $B = TU$ where
$$
T := \left\{ \: \left( \begin{array}{cc} a & 0 \\ 0 & b \end{array}
\right) \ \vert \  a , b \in \mathbb{F}^{\times}_p \: \right\}.
$$
Furthermore, we have the Bruhat decomposition
$$
GL_2(\mathbb{F}_p) = B \cup B w B \ \ \text{where} \ \ w = \left(
\begin{array}{rr} 0 & 1 \\ -1 & 0 \end{array} \right)
$$
(recall that $w$ normalizes $T$). Now, if $H_p$ is not normal in
$H$,  then it follows from the Bruhat decomposition that $H$
contains an element of the form $tw$ with $t \in T$. Consequently,
it also contains
$$
U^{-} := (tw)^{-1} U (tw) = \left\{ \: \left( \begin{array}{cc} 1 &
0 \\ a & 1 \end{array} \right) \ \vert \ a \in \mathbb{F}_p \:
\right\}.
$$
But $\langle U \: , \: U^{-} \rangle = SL_2(\mathbb{F}_p)$, and our
assertion follows.
\end{proof}

\vskip3mm

So, for any subgroup $H \subset GL_2(\mathbb{F}_p)$, we have only
the following three possibilities:

\vskip2mm

(1) $H^+ = \{ 1 \}$;

\vskip1mm

(2) $H^+$ is conjugate to $U$;

\vskip1mm

(3) $H^+ = SL_2(\mathbb{F}_p)$.

\vskip2mm

\noindent In either case, the assertion of Nori's Theorem is valid.

\vskip3mm

\noindent {\bf 3. Proof of Theorem \ref{T:Red}.} Recall that the
famous Jordan Theorem states the following:

\vskip2mm

\begin{center}

\parbox[t]{12cm}{{\it There exists a function $\mathbf{j}(n)$ on positive integers
such that if $\mathcal{G} \subset GL_n(\mathcal{K})$ is a finite
linear group over a field $K$ of characteristic zero, then
$\mathcal{G}$ contains an abelian normal subgroup $\mathcal{N}$ such
the index $[\mathcal{G} : \mathcal{N}]$ divides $\mathbf{j}(n)$.}}

\end{center}

\vskip2mm

\noindent (In a more common form, Jordan's Theorem provides a
function $\mathbf{j}_0(n)$ for which $\mathcal{G} , \mathcal{N}$ as
above satisfy $[\mathcal{G} : \mathcal{N}] \leqslant
\mathbf{j}_0(n)$; note that given such a function $\mathbf{j}_0(n)$,
the above statement holds with $\mathbf{j}(n) =
(\mathbf{j}_0(n))!$.)\footnote{Various sources give different
expressions for a Jordan function $\mathbf{j}_0(n)$; the optimal
function is known to be $\mathbf{j}_0(n) = (n + 1)!$ for $n
\geqslant 71$  - see \cite{Collins1}.} What we need to observe for
the proof of Theorem \ref{T:Red} is that the assertion of Jordan's
theorem remains valid (with the same $\mathbf{j}(n)$) for any
subgroup $\mathcal{G} \subset GL_n(\mathbb{F}_p)$ of order not
divisible by $p$.

Indeed, consider the reduction modulo $p$ map $\rho \colon
GL_n(\Z_p) \to GL_n(\mathbb{F}_p)$. The kernel $\mathrm{Ker}\: \rho
= GL_n(\Z_p , p)$ is a pro-$p$ group, so since the order of
$\mathcal{G} \subset GL_n(\mathbb{F}_p)$ is prime to $p$ there is a
section $\sigma \colon \mathcal{G} \to GL_n(\Z_p)$ for $\rho$ over
$\mathcal{G}$. Applying the standard Jordan theorem for
characteristic zero to $\tilde{\mathcal{G}} := \sigma(\mathcal{G})$,
we obtain the corresponding assertion for $\mathcal{G}$. (For the
sake of completeness, we would like to indicate that there are
various ``modular" forms of Jordan's theorem that treat finite
subgroups $\mathcal{G} \subset GL_n(\mathcal{K})$ of order divisible
by $p$ where $p = \mathrm{char}\: \mathcal{K}$, starting with
\cite{BrFeit} - see  \cite{Collins2}, \cite{Weis-J} for subsequent
results (we also note that \cite{Bass} provides a generalization to
algebraic groups). As we have already mentioned, the most general
results about finite linear groups in positive characteristic are
contained in \cite{LarPi}.)

Now, suppose that $G \subset \mathrm{GL}_n$. Let $j = \mathbf{j}(n)$
be the value of the Jordan function for this $n$. Set
$$
\Gamma^{(j)} = \langle \gamma^j \: \vert \: \gamma \in \Gamma
\rangle,
$$
and  $\Phi = [\Gamma^{(j)} , \Gamma^{(j)}]$. Since the regular map
$G \to G$, $x \mapsto x^j$, is dominant, and $G = [G , G]$, we
conclude that $\Phi$ is Zariski-dense in $G$, in particular, it is
nontrivial. Then, by expanding $\Pi$, which initially needs to be
chosen to satisfy conditions (1) and (2) of the theorem, we may
assume that for all $p \notin \Pi$ we have $\rho_p(\Phi) \neq \{ 1
\}$ where $\rho_p \colon G(\Z_p) \to \Gu^{(p)}(\mathbb{F}_p)$ is the
reduction modulo $p$ map. In addition, as we explained earlier, by
expanding $\Pi$ further, we may assume for $p \notin \Pi$, the group
$\mathrm{Ad}\: \rho_p(\Gamma)$ acts on
$\underline{\mathfrak{g}}^{(p)}$ ($=$ the Lie algebra of
$\Gu^{(p)}$) absolutely irreducibly, and also Nori's theorem is
applicable to $GL_n(\mathbb{F}_p)$. We will now show that the
resulting $\Pi$ is as required.

\vskip1mm

Let $p \notin \Pi$, and set $H = \rho_p(\Gamma) \subset
GL_n(\mathbb{F}_p)$. First, we observe that $p$ divides the order of
$H$. Indeed, otherwise by the version  of Jordan's theorem mentioned
above, there would exist an abelian normal subgroup $N \subset H$ of
index dividing $j$. Then $\rho_p(\Gamma^{(j)}) \subset N$, and
therefore $\rho_p(\Phi) = \{ 1 \}$, a contradiction. This means that
if we define $H^+$ and $\widetilde{H}$ as in the Nori's theorem,
then $\widetilde{H} \neq \{ 1 \}$, and hence the Lie algebra
$\widetilde{\mathfrak{h}}$ of $\widetilde{H}$ is a nonzero subspace
of $\underline{\mathfrak{g}}^{(p)}$. On the other hand, by our
construction, $\widetilde{H}$ is normalized by $\rho_p(\Gamma)$, so
the space $\widetilde{\mathfrak{h}}$ is $\mathrm{Ad}\:
\rho_p(\Gamma)$-invariant. Combining this with the absolute
irreducibility of the latter, we obtain that
$\widetilde{\mathfrak{h}} = \underline{\mathfrak{g}}^{(p)}$, i.e.
$\widetilde{H} = \Gu^{(p)}$. Furthermore, since $G$ is simply
connected, so is $\Gu^{(p)}$, and therefore by the affirmative
answer to the Kneser-Tits conjecture over finite fields, we have
$\Gu^{(p)}(\mathbb{F}_p) = \Gu^{(p)}(\mathbb{F}_p)^+$. Invoking
Nori's theorem, we obtain
$$
H = \widetilde{H}(\mathbb{F}_p)^+ = \Gu^{(p)}(\mathbb{F}_p)^+ =
\Gu^{(p)}(\mathbb{F}_p),
$$
as required. \hfill $\Box$

\vskip3mm

\noindent {\bf Remarks.} 1. The proof of Theorem \ref{T:Red}
sketched above is based on Nori's paper \cite{Nori}, and is
different from the original argument in \cite{MVW}. The interested
reader can find an outline of this argument (which relied on the
classification of finite simple groups in \cite[pp.
397-398]{LubSeg}.

\vskip1mm

2. Combining Lemmas \ref{L:SL2}, \ref{L:SL21} with the above
argument, we obtain a virtually complete proof of Theorem
\ref{T:MVW} for $G = \mathrm{SL}_2$, which, as we have pointed out
earlier, was essentially present already in Serre's book
\cite{Serre}.

\vskip1mm

3. It is worth pointing out that the simply connectedness of $G$ was
again used to conclude that the group $\Gu^{(p)}(\mathbb{F}_p)$ is
generated by unipotent elements. This is yet another manifestation
of the connection between strong approximation and the Kneser-Tits
conjecture that was first pointed out by Platonov \cite{Platonov}.

\vskip1mm

4. During the workshop, I.~Rivin asked if one can give an explicit
bound $N = N(\Gamma)$ such that for any $p > N$ we have
$\rho_p(\Gamma) = \underline{G}^{(p)}(\mathbb{F}_p)$. In  ongoing
work with my student A.~Morgan, we have been able to produce some
bounds of this kind. More precisely, for $g = (g_{ij}) \in
SL_n(\Z)$, set $$\no g \no = \max_{i,j} \vert g_{ij} \vert.$$
Furthermore, given a Zariski-dense subgroup $\Gamma = \langle g_1,
\ldots , g_d \rangle \subset SL_n(\Z)$, set $$\displaystyle m =
\max_{k = 1, \ldots , d} \no g_k \no.$$ Then there exists $N = N(d,
m, n)$ such that for any prime $p > N$ we have $\rho_p(\Gamma) =
SL_n(\mathbb{F}_p)$. However, at the time of this writing our bounds
are too large to be of practical use.

\vskip3mm

\noindent {\bf 4. Weisfeiler's theorem.} A far-reaching
generalization of Theorem \ref{T:MVW} was given by B.~Weisfeiler
\cite{Weis}. We will state his result using the original notations
(which are somewhat different from the notations used in the rest of
our article).
\begin{thm}\label{T:Weis}
{\rm (\cite{Weis})} Let $k$ be an algebraically closed field of
characteristic different from 2 and 3, and let $G$ be an almost
simple, connected and simply connected algebraic group defined over
$k$. Let $\Gamma$ be a Zariski-dense finitely generated subgroup of
$G(k)$, and let $A$ be the subring of $k$ generated by the traces
$\mathrm{tr}\: \mathrm{Ad}\: \gamma$, $\gamma \in \Gamma$. Then
there exists $b \in A$, a subgroup $\Gamma' \subset \Gamma$, and
a~structure $G_{A_b}$ of a group scheme over $A_b$ on $G$ such that
$\Gamma' \subseteq G_{A_b}(A_b)$ and $\Gamma'$ is dense in
$G_{A_b}(\widehat{A_b})$.
\end{thm}

(Here $A_b$ denotes the localization of $A$ with respect to $b$, and
$\widehat{A_b}$  the profinite completion of the ring $A_b$, i.e.,
the completion with respect to the topology given by all ideals of
finite index. To connect this with our previous results, we note
that for $A = \Z$, the ring $A_b$ coincides with $\Z[\frac{1}{p_1},
\ldots , \frac{1}{p_r}]$ where $p_1, \ldots , p_r$ are the primes
dividing $b$, and the completion $\widehat{A_b}$ is precisely
$\prod_{p \notin \{p_1, \ldots , p_r\}} \Z_p$, i.e. the ring of
integral $S$-adeles for $S = \{ \infty, p_1, \ldots , p_r \}$.)

\vskip2mm

In characteristic 2 and 3, one encounters additional problems due to
the existence of so-called nonstandard isogenies. We will not get
into these technical details here, but roughly speaking one of the
problems is that in these exceptional cases the ``right" ring (or
field) of definition of $\Gamma$ may not be the trace ring (resp.,
field), i.e. the subring (subfield) of the algebraically closed
field $k$ generated by the traces $\mathrm{tr}\: \mathrm{Ad}\:
\gamma$ for $\gamma \in \Gamma$. The adequate definitions were given
by Pink \cite{Pink2} using the notion of so-called minimal triples
(which we will not discuss here). In fact, Pink's paper
\cite{Pink2},  where he proved an appropriate version of the
openness statement for the adelic closures of Zariski-dense
subgroups in all characteristics, was really the final word in the
strong approximation saga.

\vskip3mm

\noindent {\bf 5. Applications to group theory: Lubotzky's
alternative.} One of the most notable applications of strong
approximation is the so-called {\it Lubotzky alternative} for linear
groups. It is discussed in detail in \cite[Ch. II]{KNV} and
\cite[Window 9]{LubSeg}, so here we will only state it for  linear
groups over  fields of characteristic zero.
\begin{thm}
{\rm (\cite{LubMann})} Let $\Gamma$ be a finitely generated linear
group over a field of characteristic zero. Then one of the following
holds:

\vskip2mm

\noindent {\rm (a)} $\Gamma$ is virtually solvable;

\vskip1mm

\noindent {\rm (b)} \parbox[t]{11.7cm}{there exists a connected
absolutely almost simple simply connected algebraic $\Q$-group $G$,
a finite set $\Pi = \{p_1, \ldots , p_r\}$ of primes such that the
group $G(\Z_{\Pi})$, where $\Z_{\Pi} = \Z[\frac{1}{p_1}, \ldots ,
\frac{1}{p_r}]$, is infinite, and a subgroup $\Gamma_1 \subset
\Gamma$ of finite index for which the profinite completion
$\widehat{\Gamma_1}$ admits a~continuous epimorphism onto
$G(\widehat{\Z_{\Pi}})$.}
\end{thm}

This statement was applied in \cite{LubMann} to study the {\it
subgroup growth} (= number of subgroups of a given index $n$) of
linear groups; in particular, it leads to the following dichotomy:
if a linear group has polynomial subgroup growth, then it is
virtually solvable, but if the growth is not polynomial (hence the
group is not virtually solvable), then it is at least $n^{\log n}$.

The interested reader will find more group-theoretic applications of
strong approximation in  \cite{KNV}, \cite{LubSeg} and references
therein,  and, of course, in other articles contained in this
volume.

\vskip3mm

{\small \noindent {\bf Acknowledgements.} The author was partially
supported by NSF grant DMS-0965758, BSF grant 2010149 and the
Humboldt Foundation. The hospitality of the MSRI during the workshop
``Hot Topics: Thin Groups and Super-strong Approximation" (February
6-10, 2012) is also thankfully acknowledged.}

\vskip5mm

\bibliographystyle{amsplain}

\end{document}